\documentclass[12pt]{amsart}
\usepackage{graphicx}
\usepackage{a4wide}
\usepackage{amssymb}
\usepackage{amsthm}
\usepackage{amsmath}
\usepackage{amscd} \usepackage{verbatim}
\usepackage[all]{xy}
\textheight8.5in \textwidth6.7in \numberwithin{equation}{section}
\theoremstyle{plain}
\newtheorem{theorem}{Theorem}[section]

\theoremstyle{definition}

\theoremstyle{remark}
\newtheorem*{remark}{Remark}

\newcommand{\maxp}{\text {\rm maxp}}

\newcommand{\N}{\mathbb{N}}

\newcommand{\la}{\lambda}

\def\({\left(}
\def\){\right)}

\begin{document}
\title[Maximal multiplicative properties of partitions]{Maximal multiplicative properties of partitions}
\author{Christine Bessenrodt and Ken Ono}
\address{Faculty of Mathematics and Physics,
Leibniz University Hannover,
Welfengarten 1,
D-30167 Hannover,
Germany}
\email{bessen@math.uni-hannover.de}
\address{Department of Mathematics,  Emory University,
Atlanta, GA. 30322}
\email{ono@mathcs.emory.edu}
\thanks{The second author thanks the NSF and the Asa Griggs Candler Fund for their generous support.}
\begin{abstract}
Extending the partition function multiplicatively to
a function on partitions, we show that it has a unique maximum
at an explicitly given partition for any $n\neq 7$.
The basis for this is an inequality for the partition
function which seems not to have been noticed before.
\end{abstract}
\maketitle

\section{Introduction and statement of results}\label{sec:intro}

For $n\in \N$, the partition function $p(n)$ enumerates the number
of partitions of $n$, i.e., the number of integer sequences
$\la=(\la_1,\la_2,\ldots )$ with
$\la_1\geq \la_2 \geq \ldots >0$ and $\sum_{j\geq 1} \la_j =n$.
It plays a central role in many parts of mathematics
and has been for centuries an object whose properties have been
studied in particular in combinatorics and number theory.

While explicit formulae for $p(n)$ are known due to
the work of Hardy, Ramanujan and Rademacher, and the recent work
of Bruinier and the second author \cite{BruinierOno} on a finite algebraic formula,
these expressions can be quite forbidding when
one wants to check even simple properties.
In a representation theoretic context, a question
came up which led to the observation of surprisingly nice
multiplicative behavior.

In this note, we show in Theorem~\ref{thm1}
the following inequality:
\medskip

\centerline{
For any integers $a,b$ such that  $a,b >1$ and $a+b>9$, we have
$p(a)p(b)>p(a+b)$.}
\medskip

This result allows us to study an ``extended partition function",
which is obtained by defining for  a partition $\mu=(\mu_1,\mu_2,\ldots)$:
$$
p(\mu)=\prod_{j\ge 1} p(\mu_j).
$$
Let $P(n)$ denote the set of all partitions of~$n$.
Here we determine
the maximum of the partition function on $P(n)$
explicitly; more precisely, we find in Theorem~\ref{thm2}
that for $n\neq 7$, the maximal value
$$\maxp(n)=\max(p(\mu) \mid \mu \in P(n))$$
is attained at a unique partition of $n$ of a very simple form,
which depends on $n$ modulo~$4$.

\begin{theorem}\label{thm2}
Let $n\in \N$. For $n\geq 4$ and $n\neq 7$,
the maximal value $\maxp(n)$
of the partition function on $P(n)$
is attained at the partition
$$
\begin{array}{ll}
(4,4,4,\ldots,4,4) & \text{when } n \equiv 0 \pmod 4\\
(5,4,4,\ldots,4,4) & \text{when } n \equiv 1 \pmod 4\\
(6,4,4,\ldots,4,4) & \text{when } n \equiv 2 \pmod 4\\
(6,5,4,\ldots,4,4) & \text{when } n \equiv 3 \pmod{4}
\end{array}
$$
In particular, if $n\geq 8$, then
$$
\maxp(n)=\begin{cases}
5^{\frac{n}{4}}\ \ \ \ \ \ \ \
 & \text {\rm if $n\equiv 0\pmod 4$},\\
7\cdot 5^{\frac{n-5}{4}}\ \ \ \ \ \ \ \ &\text {\rm if $n\equiv 1\pmod 4$},\\
11\cdot 5^{\frac{n-6}{4}}\ \ \ \ \ \ \ \ &\text {\rm if $n\equiv 2\pmod 4$},\\
11\cdot7\cdot 5^{\frac{n-11}{4}}\ \ \ \ \ \ \ \ &\text {\rm if $n\equiv 3\pmod 4$}.
\end{cases}
$$
\end{theorem}

\section{An analytic result on the partition function}

The main result of this section is
the following analytic inequality for the partition function~$p(n)$.
\begin{theorem}\label{thm1}
If $a,b$ are integers with $a,b>1$ and $a+b>8$, then
$$
p(a)p(b)\geq p(a+b).
$$
with equality holding only for $\{a,b\}= \{2,7\}$.
\end{theorem}

\begin{remark} Of course, the inequality in Theorem~\ref{thm1} always fails if we
take $a=1$. The complete set of pairs of integers $1<a\leq b$ for which
the inequality fails is
$$
\left \{ (2,2), (2,3), (2,4), (2,5), (3,3),  (3,5)\right \},
$$
while for
$$
\left \{ (2,6), (3,4)\right \}
$$
we have equality.
\end{remark}

The main tool for deriving Theorem~\ref{thm1} is the following classical result
of D. H. Lehmer~\cite{Lehmer}.

\begin{theorem}[Lehmer]\label{LehmerThm}
If $n$ is a positive integer and $\mu=\mu(n):=\frac{\pi}{6}\sqrt{24n-1}$,
then
$$
p(n)=\frac{\sqrt{12}}{24n-1}\cdot \left [\left(1-\frac{1}{\mu}\right)e^{\mu}+
\left(1+\frac{1}{\mu}\right)e^{-\mu}\right]+E(n),
$$
where we have that
$$
|E(n)|<\frac{\pi^2}{\sqrt{3}}\cdot \left [ \frac{1}{\mu^3}\sinh(\mu)+\frac{1}{6}-\frac{1}{\mu^2}\right].
$$
\end{theorem}

\begin{proof}[Proof of Theorem~\ref{thm1}]
By Theorem~\ref{LehmerThm}, it is straightforward to verify for every positive integer~$n$ that
$$
\frac{\sqrt{3}}{12n}\left(1-\frac{1}{\sqrt{n}}\right)e^{\mu(n)}<p(n)<
\frac{\sqrt{3}}{12n}\left(1+\frac{1}{\sqrt{n}}\right)e^{\mu(n)}.
$$
We may assume $1<a\leq b$;
for convenience we let $b=\lambda a$, where $\lambda\geq 1$.
These inequalities immediately give:
\begin{displaymath}
\begin{split}
p(a)p(\lambda a) & > \frac{1}{48\lambda a^2}\left(1-\frac{1}{\sqrt{a}}\right)\left(1-
\frac{1}{\sqrt{\lambda a}}\right)\cdot e^{\mu(a)+\mu(\lambda a)},\\
p(a+\lambda a) & < \frac{\sqrt{3}}{12(a+\lambda a)}\left(1+\frac{1}{\sqrt{a+\lambda a}}\right)
e^{\mu(a+\lambda a)}.
\end{split}
\end{displaymath}
For all but finitely many cases, it suffices to find conditions on $a>1$ and $\lambda\geq 1$
for which
$$
\frac{1}{48\lambda a^2}\left(1-\frac{1}{\sqrt{a}}\right)\left(1-\frac{1}{\sqrt{\lambda a}}\right)e^{\mu(a)+\mu(\lambda a)}>
\frac{\sqrt{3}}{12(a+\lambda a)}\left(1+\frac{1}{\sqrt{a+\lambda a}}\right)e^{\mu(a+\lambda a)}.
$$
Since $\lambda\geq 1$, we have that $\lambda/(\lambda+1)\geq 1/2$, and so it suffices to consider when
$$
e^{\mu(a)+\mu(\lambda a)-\mu(a+\lambda a)}>
2a\sqrt{3}\cdot S_a(\lambda),
$$
where
\begin{equation}\label{Sa}
S_a(\lambda):=\frac{1+\frac{1}{\sqrt{a+\lambda a}}}{\left(1-\frac{1}{\sqrt{a}}\right)\left(1-\frac{1}{\sqrt{\lambda a}}\right)}.
\end{equation}
By taking the natural log, we obtain the inequality
\begin{equation}\label{condition}
T_a(\lambda)> \log(2a\sqrt{3})+\log(S_a(\lambda)),
\end{equation}
where
we have that
\begin{equation}\label{Ta}
T_a(\lambda):=\frac{\pi}{6}\left(\sqrt{24a-1}+\sqrt{24\lambda a-1}-\sqrt{24(a+\lambda a)-1}\right).
\end{equation}
We consider (\ref{Sa}) and (\ref{Ta}) as functions in $\lambda\geq 1$ and fixed $a>1$. Simple calculations reveal
that $S_a(\lambda)$ is decreasing in $\lambda\geq 1$, while $T_a(\lambda)$ is increasing in $\lambda\geq 1$.
Therefore, (\ref{condition}) becomes
$$
T_a(\lambda)\geq T_a(1) > \log(2a\sqrt{3}) +\log(S_a(1))\geq \log(2a\sqrt{3})+\log(S_a(\lambda)).
$$
By evaluating $T_a(1)$ and $S_a(1)$ directly, one easily finds that (\ref{condition}) holds whenever $a\geq 9$.
To complete the proof, assume that $2\leq a \leq 8$. We then directly calculate the real number $\lambda_a$ for which
$$
T_a(\lambda_a) = \log(2a\sqrt{3})+\log(S_a(\lambda_a)).
$$
By the discussion above, if $b=\lambda a\geq a$ is an integer for which $\lambda > \lambda_a$, then (\ref{condition}) holds,
which in turn gives the theorem in these cases.
The table below gives the numerical calculations for these $\lambda_a$.
\medskip
\begin{table}[h]
\caption{\label{table:2} Values of $\lambda_a$}
\begin{tabular}{|r|cc| }
\hline \rule[-3mm]{0mm}{8mm}
$a$ && $\lambda_a$\\
\hline% \rule[-3mm]{0mm}{8mm}
$2$  && $57.08\dots$ \\%26731$\\
$3$  && $7.42\dots$ \\%69658$\\
$4$  && $3.62\dots$ \\%04156$\\
$5$ && $2.36\dots$ \\%73346$\\
$6$&&  $1.74\dots$ \\%50497$\\
$7$&& $1.38\dots$\\
$8$&& $1.15\dots$ \\%74948$\\
\hline
\end{tabular}
\end{table}
\medskip
\noindent
Only finitely many cases remain, namely the pairs of integers
where $2\leq a\leq 8$ and $1\leq b/a \leq \lambda_a$.
We computed $p(a), p(b)$ and $p(a+b)$ for these
cases to complete the proof of the theorem.
\end{proof}

\section{The maximum property}\label{sec:max}

Here we use the result in the previous section to prove Theorem~\ref{thm2}.

\begin{proof}[Proof of Theorem~\ref{thm2}]
For the proof, we will need the partitions where $\maxp(n)$ is attained
for $n\leq 14$;
these are given in Table~2 below (computed by Maple).
\medskip

\begin{table}[h]
\caption{\label{table:4} Maximum value partitions $\mu$}
\begin{tabular}{|r|cc|cc|cc| }
\hline \rule[-3mm]{0mm}{8mm}
$n$ && $p(n)$ && \maxp(n) && $\mu$ \\
\hline% \rule[-3mm]{0mm}{8mm}
$1$  && 1 && 1 && (1) \\
$2$  && 2 && 2 && (2) \\
$3$  && 3 && 3 && (3) \\
$4$  && 5 && 5 && (4) \\
$5$ && 7 && 7 && (5) \\
$6$&&  11 && 11 && (6) \\
$7$&& 15 && 15 && (7), (4,3)\\
$8$&& 22 && 25 && (4,4) \\
$9$&& 30 && 35 && (5,4) \\
$10$&& 42  && 55 && (6,4) \\
$11$&& 56 && 77 && (6,5) \\
$12$&& 77 && 125 && (4,4,4) \\
$13$&& 101 && 175 && (5,4,4) \\
$14$&& 135 && 275 && (6,4,4) \\
\hline
\end{tabular}
\end{table}
\medskip

\noindent
We see that the assertion holds for $n\leq 14$, and we may thus assume now that $n>14$.
Let $\mu=(\mu_1,\mu_2,\ldots)\in P(n)$ be such that $p(\mu)$ is maximal.
If $\mu$ has a part $k\geq 8$, then by Theorem~\ref{thm1} and Table~2,
replacing $k$ by the parts $\lfloor \frac k2 \rfloor, \lceil \frac k2 \rceil$ would produce a partition $\nu$ such that $p(\nu)>p(\mu)$.
Thus all parts of $\mu$ are smaller than~$8$.
Let $m_j$ be the multiplicity of a part $j$ in $\mu$.
If $m_1\ne 0$, then for $\nu=(\mu_1+1,\mu_2,\ldots)$
we have $p(\nu)>p(\mu)$.
So $m_1=0$.
If $m_2\geq 2$, then replacing parts $2,2$ in $\mu$ by one part $4$ gives a partition $\nu$ with $p(\nu)>p(\mu)$. So $m_2\le 1$.
Similarly, the operations of replacing $(3,3)$ by a part $6$,
$(5,5)$ by the parts $(6,4)$, $(6,6)$ by the parts $(4,4,4)$,
and $(7,7)$ by the parts $(6,4,4)$, respectively,
show that we have $m_3, m_5, m_6, m_7 \le 1$.

Now assume that $m_7=1$.
As $n>14$, one of $m_2, m_3, m_4, m_5, m_6$ is nonzero;
now by  Table~2, performing one of the following replacement operations
$$
7,2 \to 5,4; \; 7,3 \to 6,4;\; 7,4 \to 6,5;\;  7,5 \to 4,4,4;\; 7,6 \to 5,4,4
$$
gives a partition $\nu$ with $p(\nu)>p(\mu)$, a contradiction.
Hence $m_7=0$.

Next assume that $m_6=1$.
If $m_2=1$ or $m_3=1$, we can use the following operations that increase
the $p$-value:
$$ 6,2 \to 4,4; \; 6,3 \to 5,4 \:.$$
By the choice of $\mu$, we conclude that $m_2=0=m_3$.
Thus in this case we have
$\mu=(6,4,4,\ldots,4,4)$ or $\mu=(6,5,4,4,\ldots,4,4)$,
which is in accordance with the assertion.

Thus we may now assume $m_6=0$.
Assume first that $m_5=1$. Note that we must have a part 4, since $n>14$.
If $m_2=1$ or $m_3=1$, we can increase
the $p$-value by the replacements:
$$ 5,4,2 \to 6,5;\; 5,3 \to 4,4 \:.$$
By the choice of $\mu$, this implies $m_2=0=m_3$.
Thus we have in this case
$\mu=(5,4,4,\ldots,4,4)$, again in accordance with the assertion.

Now we consider the case where also $m_5=0$.
If $m_2=1$ or $m_3=1$, we use the replacements
$$ 4,2 \to 6; \; 4,4,3 \to 6,5 $$
to get a contradiction.
Hence $m_2=0=m_3$ and we have the final case
$\mu=(4,4,\ldots,4,4)$ occurring in the assertion.

The four types of partitions we have found occur at
the four different congruence classes of~$n \mod 4$;
thus for each value of $n\neq 7$, we have found
a unique partition $\mu$ that maximizes the $p$-value
and we are done.
\end{proof}

\section{Concluding remarks}

We note that recently also other multiplicative properties of the partition function have been studied.
Originating in a conjecture by William Chen, DeSalvo and Pak have proved
log-concavity for the partition function for all $n>25$;
indeed, they have shown that for all $n>m>1$ the following holds:
$$p(n)^2 > p(n-m)p(n+m)\:.$$
Note that the border case $m=n$ (which is not covered here)
is included in our results, for any $n>3$.
\medskip

The opposite inequality
$$p(1)p(n) < p(n+1)$$
has an easy combinatorial proof by an injection $P(n) \to P(n+1)$.
One may ask whether there is also a combinatorial argument for
proving Theorem~\ref{thm1}.
\medskip

The behavior that we have seen here for the partition function $p(n)$
seems to be shared also by the enumeration of other
sets of suitably restricted partitions; work on this is in progress.

\end{document}